\documentclass[10pt]{article}
\usepackage[square,authoryear]{natbib}
\usepackage{marsden_article}
\usepackage{framed}

\begin{document}

\title{Symmetric Discrete Optimal Control and Deep Learning}

\author{Anthony M. Bloch\thanks{
Research partially supported by NSF grant  DMS-2103026, and AFOSR grants FA
9550-22-1-0215 and FA 9550-23-1-0400.}
            \\Department of Mathematics
            \\ University of Michigan \\ Ann Arbor MI 48109
            \\{\small abloch@math.lsa.umich.edu}
          \and
Peter E. Crouch
            \\ College of Engineering
             \\ University of Texas at Arlington 
            \\Arlington, TX 
            \\{\small peter.crouch@uta.edu}
             \and
Tudor S. Ratiu\thanks{Research partially supported 
by the National Natural Science Foundation of China  grant
number 11871334 and by the Swiss National Science Foundation grant 
NCCR SwissMAP. }
\\ School of Mathematical Sciences
\\ Ministry of Education Laboratory of Scientific Computing (MOE-LSC)
\\ Shanghai Frontier Science Center of Modern Analysis
\\ Shanghai Jiao Tong University
\\ 800 Dongchuan Road, 200240 China \\  
Section de Math\'ematiques 
\\ Ecole Polytechnique F\'ed\'erale de 
Lausanne \\ 1500 Lausanne, Switzerland
\\{\small ratiu@sjtu.edu.cn, tudor.ratiu@epfl.ch}
}

            \date{\small February 8,  2024\\}
            \maketitle

\centerline{\it In memory of Roger Brockett}

\begin{abstract}
We analyze discrete optimal control problems and their connection 
with back propagation and deep learning.  We consider in particular 
 the symmetric representation of the discrete rigid body equations 
developed via optimal control  analysis and optimal flows on adjoint orbits 

\end{abstract}
\bigskip


%\tableofcontents

\section{Introduction}\label{Intro}

This paper presents some connections between discrete optimal control, 
deep learning, and back propagation.  This goes back  to the work 
of \cite{BrHo1969} on discrete optimal control and, in particular,  
to the paper \cite{LeCun1988}. We show that  the formalism developed in 
\cite{BlCrMaRa2002} has much in common with this latter paper. 
Other interesting recent work on deep learning and optimal control includes  \cite{BeCeEhOwSc2019}, \cite{E2017}, \cite{E2019} and references therein.

We  consider here firstly the discrete setting and discuss also  the 
smooth setting and double bracket equations. 
We consider in particular the discrete symmetric rigid body equations 
developed in \cite{BlCrMaRa2002}. 

The key differences between this work and our earlier work 
\cite{BlCrMaRa2002} is, as  appropriate for machine learning,
a terminal cost rather than a fixed endpoint condition and 
multiple sets of initial data.  The connections
between deep learning and optimal control  are essentially
the following: the controls are the weights, the layers are 
the discrete time points, the training data or patterns
are the initial data, back propagation is solving the 
two point boundary problem, and the endpoint cost is the primary cost.

This is related to our earlier work on an alternative formulation 
of the $n$-dimensional rigid body equations and 
an associated set of discrete equations called the {\it symmetric 
representation of the discrete rigid body equations}; see 
\cite{BlCrMaRa2002} and  \cite{GuBl2004}. Both the continuous equations and their discrete counterpart evolve on a Cartesian product 
$G \times G $ of a Lie group $G$ rather than on its cotangent 
bundle $T^\ast G$. One interesting way to derive the continuous 
equations is by means of the (Pontryagin) maximum principle of 
optimal control theory. Likewise, the discrete equations can be 
derived from discrete optimal control theory. Extensions of the 
theory in the smooth setting may be found, for example, in 
\cite{BlCrMaSa2008}, \cite{GaRa2011}, and \cite{BlGaRa2018}. 

It is also interesting to consider the subRiemannian setting 
of these problems as we did with a view to the quantum 
setting in \cite{SaMoBl2009}.

The work described in this paper is in part expository and in part new. 
It owes much to earlier work by and with Roger Brockett, see, e.g. 
\cite{Brockett1973}, \cite{Brockett1989}, \cite{Brockett1994}, 
\cite{BlBrRa1990}, \cite{BlBrRa1992}, \cite{BlBrCr1997}. 

\section{Discrete Optimal Control and Associated  Deep Learning
}\label{discreteoptimal.section}

We first consider a general  class of discrete optimal
control problems and we follow with the special case of the discrete symmetric representation of the rigid body equations

\noindent
\begin{definition}\label{prob8.1}
Let $N$ be a positive integer and $X_0,X_N\in\mathbb{R}^n$ be given. 
 Let $f:\mathbb{R}^n\times\mathbb{R}^m \rightarrow  
\mathbb{R}^n$ and $g:\mathbb{R}^n\times\mathbb{R}^m \rightarrow
\mathbb{R}^+:=[0,\infty)$ be smooth functions. Denote points 
in $\mathbb{R}^n$ by $x$, points in $\mathbb{R}^m$ by $u$. 
Let $E\subset\mathbb{R}^m$ be a control constraint set and
assume that $E = h^{-1}(0)$, where 
$h:\mathbb{R}^m\to\mathbb{R}^l$ is a submersion.
Let $\langle\cdot,\cdot\rangle$ denote the pairing between vectors in
$\mathbb{R}^p$ given by the dot product; $p$ can be $n$, $m$, or $l$.

Define the optimal control problem:

\begin{equation}\label{eq8.1}
\min_{u_k\in E}\sum^{N-1}_{k=0}g(x_k,u_k)
\end{equation}

subject to
$x_{k+1}=f(x_k,u_k)$, with $x_0=X_0$ and $ x_N=X_N$,
for $u_k\in E$.
\end{definition}

\begin{proposition}\label{gendiscprop}
A solution to problem (\ref{prob8.1}) satisfies the following extremal
solution equations
\begin{equation} \label{eq8.3}
p_k = \frac{\partial H}{\partial x_k}(p_{k+1},x_k,u^*_k), \qquad
x_{k+1} = \frac{\partial H}{\partial p_{k+1}}(p_{k+1},x_k,u^*_k),
\end{equation}
  where
\begin{align} \label{eq8.4}
H(p_{k+1},x_k,u_k):=\langle p_{k+1},f(x_k,u_k)\rangle-g(x_k,u_k),
\quad 
 k=0, \ldots, N, \quad p_{N+1} =0,
\end{align}
and $\partial/ \partial x$, $\partial/ \partial p$ denote the partial 
$x$- and $p$-gradients.
In these equations, $u_k ^\ast$ is determined as follows. Define
\begin{equation}\label{eq8.5}
\hat{H}(p_{k+1},x_k,u_k,\sigma):=H(p_{k+1},x_k,u_k)+ 
\left\langle \sigma,h(u_k)\right\rangle
\end{equation}
for $\sigma\in\mathbb{R}^l$,
$\sigma$ a Lagrange multiplier;
then $u^*_k$ and $\sigma$ are solutions of the equations
\begin{equation}\label{eq8.6}
\frac{\partial\hat{H}}{\partial u_k}(p_{k+1},x_k,u^*_k,\sigma)=0,\quad
h(u^*_k)=0.
\end{equation}
\end{proposition}

\begin{proof}
Use the discrete maximum principle.
We wish to minimize $\sum^{N-1}_{k=0}g(x_k,u_k)$ subject to
the given discrete dynamics and control constraint set $E$.
To implement the constraints defining $E$, we consider
maximizing the augmented functional (which is 
independent of $p_0$) 
\begin{align*}
&V(p_{0},\ldots, p_N, x_0,\ldots x_{N-1},u_0, \ldots, u_{N-1})  \\
& \qquad :=
\sum^{N-1}_{k=0}\left( \langle
p_{k+1},f(x_k,u_k)-x_{k+1}\rangle+
\left\langle \sigma ,h(u_k)\right\rangle 
-g(x_k,u_k) \right) \\
&\qquad\, = \sum^{N-1}_{k=0}\left( -\langle p_{k+1},x_{k+1}\rangle+
\hat{H}(p_{k+1},x_k,u_k,\sigma) \right) \\
&\qquad\, = \left\langle p_0, x_0 \right\rangle +
\sum^N_{k=0} \left( -\langle 
p_k,x_k\rangle+\hat{H}(p_{k+1},x_k,u_k,\sigma)\right),
\end{align*}
where we set $p_{N+1}=0$ for notational convenience since
there is no $x_{N+1}$.
 The critical points of $V$ are hence given by
\begin{equation*}
0=
\delta V  = \left\langle \delta p_0, x_0 \right\rangle +
\sum^N_{k=0}\left( -\langle p_k,\delta x_k\rangle-\langle\delta
p_k,x_k\rangle+\frac{\partial \hat{H}}{\partial p_{k+1}}\delta p_{k+1} 
+\frac{\partial\hat{H}}{\partial x_k}\delta x_k+
\frac{\partial\hat{H}}{\partial u_k}\delta u_k \right) 
\end{equation*}

since $\delta x_0 = 0$ because $x_0 = X_0$ is a given constant vector.
This gives the extremal equations (\ref{eq8.3}) and (\ref{eq8.6}) since
\begin{align*} 
%\label{eq8.7}
\frac{\partial\hat{H}}{\partial
p_{k+1}}(p_{k+1},x_k,u_k,\sigma)&= 
\frac{\partial H}{\partial p_{k+1}}(p_{k+1},x_k,u_k),\nonumber\\
\frac{\partial\hat{H}}{\partial x_k}(p_{k+1},x_k,u_k,\sigma)&=
\frac{\partial H}{\partial x_k}(p_{k+1},x_k,u_k),
\end{align*}
and $h(u_k)=0$,  for $k=0, \ldots , N$ and $p_{N+1}=0$.  
\end{proof}

Note that for the algorithm described by equation \eqref{eq8.3}  
to make sense we need to able to compute $p_{k+1}$ from the given implicit form. 
 This follows if the $n \times  n$ matrix with entries
$\frac{\partial ^2 H}{\partial y_i \partial x_j} (y,x,u)$
for $ i,j = 1, \ldots, n,$
is invertible at every 
$(y,x,u) \in \mathbb{R}^n \times  \mathbb{R}^n \times  \mathbb{R}^m$. We need such 
a condition to be satisfied for any given algorithm.

We assume that both $u^*_k$ and $\sigma$  are determined uniquely by equations
(\ref{eq8.6}).
Also note that
$u^*_k=u^*_k(p_{k+1},x_k)$, $\sigma=\sigma (p_{k+1},x_k)$. 
 Using this hypothesis, we have the following consequence.

\begin{corollary}\label{C:8.3}
Assume that the extremal equations (\ref{eq8.3}) hold. Then
\begin{align*}
%\label{eq8.10}
dp_k&= \frac{\partial^2\hat{H}}{\partial x^2_k}(p_{k+1},x_k,u^*_k)dx_k+
\frac{\partial^2\hat{H}}{\partial p_{k+1}\partial x_k}
(p_{k+1},x_k,u^*_k)dp_{k+1}\,,\\
dx_{k+1}&= \frac{\partial^2\hat{H}}{\partial p_{k+1}\partial x_k}
(p_{k+1},x_k,u^*_k) dx_k+
\frac{\partial^2\hat{H}}{\partial p^2_{k+1}}(p_{k+1},x_k,u^*_k)
dp_{k+1}\,.
\end{align*}
\end{corollary}

We remark that the implicit advance map
$
\Phi:(x_k,p_k) \mapsto (x_{k+1},p_{k+1})
$
generated by the extremal evolution (\ref{eq8.3}) is symplectic, i.e.,
\begin{equation}\label{eq8.13}
\Phi^*(dx_{k+1}\wedge dp_{k+1})=dx_k\wedge dp_k.
\end{equation}
This is easily demonstrated 
by using Corollary \ref{C:8.3}.
One can also derive symplecticity directly from Hamilton's phase
space principle; see \cite{MaWe2001}.

We note that solving the above two point boundary value in practice in highly nontrivial. Various techniques 
have been employed including so-called shooting methods (see .e.g. \cite{BrHo1969}) and this is still an ongoing topic 
of research.

Now we modify this problem to include an endpoint cost and a form 
suitable for machine learning. 

\noindent
\begin{definition}\label{prob8.1m}
Let $N, M$ be a positive integers and $X_0,X_N^a\in\mathbb{R}^n$,
$a=1, \ldots, M$ be given. Let $f:\mathbb{R}^n\times\mathbb{R}^m
\times  \mathbb{R}^n \rightarrow \mathbb{R}^n$, 
$g:\mathbb{R}^n\times\mathbb{R}^m \rightarrow
\mathbb{R}^+:=[0,\infty)$, and $\phi: \mathbb{R}^n \rightarrow 
\mathbb{R}^+$ be smooth functions. 
Let $E\subset\mathbb{R}^m$ be a control constraint set and
assume that $E = h^{-1}(0)$, where 
$h:\mathbb{R}^m\to\mathbb{R}^l$ is a submersion.

Define the optimal control problem:

\begin{equation}\label{eq8.1m}
\min_{u_k\in E}\sum_{a=1}^M\sum^{N-1}_{k=0}g(x^a_k,u_k) +\sum_{a=1}^M\phi(x^a_N)
\end{equation}

subject to $x^a_{k+1}=f^a(x_k,u_k):=f(x_k,u_k,X_0^a)$ with 
$x_0=X^a_0$ and $x^a_N = X_N^a$, for 
$u_k\in E$, $k=0, \ldots, N-1$, 
and $a=1,\ldots, M$.
\end{definition}

The proof of the next proposition is analogous to that of
Proposition \ref{gendiscprop}.

\begin{proposition}\label{gendiscpropm}
A solution to problem (\ref{eq8.1m}) satisfies the following extremal
solution equations
\begin{equation} \label{eq8.3m}
p^a_k = \frac{\partial H}{\partial x_k}(p^a_{k+1},x_k,u^*_k), 
\qquad
x^a_{k+1} = \frac{\partial H}{\partial p_{k+1}}(p^a_{k+1},x_k,u^*_k),
\end{equation}
where $k=0,\ldots, N-1$, $p^a_{N+1}=0$ for all $a=1,\ldots, M$,
 and
\begin{align} \label{eq8.4m}
H(p_{k+1},x_k,u_k)=\sum_{a=1}^M\left(
\langle p^a_{k+1},f^a(x_k,u_k)\rangle-g(x^a_k,u_k)
-\phi(x^a_N)\right).
\end{align}
In these equations, $u _k ^\ast$ is determined as follows. Let
\begin{equation}\label{eq8.5m}
\hat{H}(p_{k+1},x_k,u_k,\sigma):=
\sum_{a=1}^M\left(
\langle p^a_{k+1},f^a(x_k,u_k)\rangle-g(x^a_k,u_k)\right)+
\left\langle \sigma,h(u_k)\right\rangle
\end{equation}
for $\sigma\in\mathbb{R}^l$, $\sigma$ a Lagrange multiplier.
Then $u^*_k$ and $\sigma$ are solutions of the equations
\begin{equation}\label{eq8.6m}
\frac{\partial\hat{H}}{\partial u_k}(p^a_{k+1},x^a_k,u^*_k,\sigma)=0,\quad
h(u^*_k)=0.
\end{equation}

In addition we have the endpoint condition

\begin{equation}\label{m8}
p^a_N=\frac{\partial\phi(x^a_N)}{\partial x^a_N}, \quad
a=1, \ldots, M.
\end{equation}

\end{proposition}

\begin{proof}
Use the discrete maximum principle.
We wish to minimize $\sum_{a=1}^M\sum^{N-1}_{k=0}g(x^a_k,u_k)+\sum_{a=1}^M\phi(x^a_N)$ subject to
the given discrete dynamics and control constraint set $E$.
To implement the constraints defining $E$, we consider
maximizing the augmented functional (which is 
independent of $p_0$)
\begin{align*}
&V(p_{0},\ldots, p_N, x_0,\ldots x_{N-1},u_0, \ldots, u_{N-1})  \\
& \qquad :=
\sum^{N-1}_{k=0}\sum^{M}_{a=1}\left( \langle
p^a_{k+1},f^a(x^a_k,u_k)-x^a_{k+1}\rangle+
\left\langle \sigma ,h(u_k)\right\rangle 
-g(x^a_k,u_k) \right)   -\sum_{a=1}^M(p^a(N)x^a(N)-\phi(x^a_N))\\
&\qquad\, = \sum_{a-1}^M\sum^{N-1}_{k=0}\left( -\langle p^a_{k+1},x^a_{k+1}\rangle+
\hat{H}(p^a_{k+1},x^a_k,u_k,\sigma) \right)  -\sum_{a=1}^M(p^a(N)x^a(N)-\phi(x^a_N))\\
&\qquad\, =\sum_{a=1}^M\left( \left\langle p^a_0, x^a_0 \right\rangle +
\sum^N_{k=0} ( -\langle 
p^a_k,x^a_k\rangle+\hat{H}(p^a_{k+1},x^a_k,u_k,\sigma)) \right) -\sum_{a=1}^M(p^a(N)x^a(N)-\phi(x^a_N))\\
\end{align*}
where we set $p^a_{N+1}=0$ for notational convenience since
there is no $x_{N+1}$.
 The critical points of $V$ are hence given by
\begin{align*}
0=
\delta V  =& \left\langle \delta p_0, x_0 \right\rangle +
\sum^N_{k=0}\left( \sum_{a=1}^M(-\langle p^a_k,\delta x^a_k\rangle-\langle\delta
p^a_k,x^a_k\rangle+\frac{\partial \hat{H}}{\partial p^a_{k+1}}\delta p^a_{k+1} 
+\frac{\partial\hat{H}}{\partial x^a_k}\delta x^a_k)+
\frac{\partial\hat{H}}{\partial u_k}\delta u_k \right) \\
&+\sum_{a=1}^M\left(\delta p^a(N)x^a(N)+p^a(N)\delta x^a(N)-\frac{\partial\phi(x^a(N)}{\partial x^a(N)}\delta x^a(N)\right),
\end{align*}

since $\delta x_0 = 0$ because $x_0 = X_0$ is a given constant vector.
This gives the extremal equations (\ref{eq8.3m}), (\ref{eq8.6m}) and (\ref{m8}) since
\begin{align*} 
%\label{eq8.7}
\frac{\partial\hat{H}}{\partial
p_{k+1}}(p_{k+1},x_k,u_k,\sigma)&= 
\frac{\partial H}{\partial p_{k+1}}(p_{k+1},x_k,u_k),\nonumber\\
\frac{\partial\hat{H}}{\partial x_k}(p_{k+1},x_k,u_k,\sigma)&=
\frac{\partial H}{\partial x_k}(p_{k+1},x_k,u_k),
\end{align*}
and $h(u_k)=0$, for $k=0, \ldots , N$ and $p_{N+1}=0$ and $p_N$ is fixed
\end{proof}

\paragraph{Remark} 1. As described in  \cite{BeCeEhOwSc2019}, a 
common choice for $f$ is $f(x,u)=\sigma(Kx+\beta)$, where $u=(K,\beta)$ 
and $sigma$ is the sigmoid function. This is the so-called ResNet 
framework.  We can, of course, consider other problems of this type 
but here we are interested  in a particular 
group theoretic form. 

2. The form  of the solution in Proposition \ref{gendiscpropm}
is very close to that of \cite{LeCun1988} and, at least on a 
superficial level, even more so in the rigid body case. 

\color{black}

\subsection {The discrete symmetric rigid body}

We now derive the discrete symmetric rigid body 
equations by considering discrete optimal control on the special orthogonal group.

\begin{definition}\label{mvoptprobm}
Let $\Lambda$ be a positive definite diagonal matrix. Let
$\overline{Q}_0, \overline{Q}_N\in \operatorname{SO}(n)$ be given and fixed.
Let

\begin{equation}
\hat{V}(U_0, \ldots, U_{N-1}):=\sum_{k=0}^{N-1}
\operatorname{trace}(\Lambda U_{k}),
\quad U_k \in  \operatorname{SO}(n).
\label{discrbopt}
\end{equation}

Define the optimal control problem

\begin{equation}
\mathop{\rm min}_{U_k\in\operatorname{SO}(n)}\hat{V}(U_0,\ldots,  U_{N-1})=
\mathop{\rm min}_{U_k\in\operatorname{SO}(n)}
\sum_{k=0}^{N-1}\operatorname{trace}(\Lambda U_{k})
\label{discrbopt2}
\end{equation}

subject to dynamics and initial and final data
\begin{equation}
Q_{r+1}=Q_kU_{r},
\qquad Q_0=\overline{Q}_0, \qquad Q_N =\overline{Q}_N
\label{discrbeq}
\end{equation}
for $Q_r, U_r\in \operatorname{SO}(n)$, $r=0,1, \ldots, N-1$.
\end{definition}

\begin{theorem}
A solution of the optimal control problem in Definition \ref{mvoptprobm}
satisfies the optimal evolution  equations 
\begin{equation}
Q_{k+1} = Q_kU_{k}\,, \qquad
P_{k+1} = P_kU_{k}\,, \qquad  k=0, \ldots, N-1,
\label{discrbopteqns}
\end{equation}
where $P_k\in \operatorname{SO}(n)$ is the discrete covector in 
the discrete maximum principle and
$U_{k} \in \operatorname{SO}(n)$ is defined by
%-----------------------------

\begin{equation}
U_{k}\Lambda - \Lambda U^T_{k}=Q_k^TP_k-P_k^TQ_k\,.
\label{Ukdef}
\end{equation}
%-----------------------------
\end{theorem}

Equation (\ref{Ukdef}) can be solved for $U_k$ under certain
circumstances, as discussed in \cite{MoVe1991} and \cite{CaLe2001}; we
discuss this issue further below.

\begin{proof}
Applying Proposition \ref{gendiscprop}, we get
%-----------------------------
\begin{equation}
H(P_{k+1},Q_k,U_{k})=\operatorname{trace}(P_{k+1}^TQ_kU_{k})
-\operatorname{trace}(\Lambda U_{k})
=\operatorname{trace}\left((P^T_{k+1}Q_k-\Lambda)U_{k}\right)\,.
\label{discham}
\end{equation}
Note that 
\[
\hat{V}(U_0, \ldots, U_{N-1})=
\sum_{k=0}^{N-1}\operatorname{trace}(\Lambda U_{k})=
\sum_{k=0}^{N-1}\operatorname{trace}(U^T_{k}
\Lambda)=\sum_{k=0}^{N-1}\operatorname{trace}(Q_k\Lambda Q_{k+1}^T)
\] is the
Moser-Veselov functional \cite{MoVe1991} and that this functional is {\it linear} in 
the controls.

We need to find the critical points of $H(P_{k+1},Q_k,U_{k})$
where $U_k^TU_k=I$ since $U_k\in \operatorname{SO}(n)$.
Thus, we need to minimize a functional of the form
  $\operatorname{trace}(AU)$, $A$ fixed,  subject to $U^TU=I$.
Set
\[
\tilde{V}(U):=\operatorname{trace}(AU)+\tfrac{1}{2}
\operatorname{trace}\left(\Sigma (U^TU-I)\right)\,,
\]
where $U \in  \operatorname{SO}(n)$ and $\Sigma=\Sigma^T$ is 
a $N \times N$ matrix of Lagrange multipliers. Then
$\delta\tilde{V}(U) \cdot \delta U=
\operatorname{trace}(A\delta U+\Sigma U^T\delta U)=0$
implies $A+\Sigma U^T=0$ where $U^TU=I$.
Hence $\Sigma=-AU$. But since $\Sigma=\Sigma^T$ the extrema of our
optimization problem are obtained when
$AU=U^TA^T$. Applying this observation to our case (see \eqref{discham}),  
we have $\nabla_{U_k}H = 0$ when
\[
\left(P_{k+1}^TQ_k-\Lambda\right)U_{k}
=U^T_{k}\left(Q_k^TP_{k+1}-\Lambda\right)\,,\]
that is,
\[
U^T_{k}\Lambda-\Lambda U_{k}=
U^T_{k}Q^T_kP_{k+1}-P^T_{k+1}Q_kU_{k}\]
or, equivalently,
\begin{equation}
U_{k}\Lambda-\Lambda U^T_{k}=-Q^T_kP_{k+1}U^T_{k}+
U_{k}P^T_{k+1}Q_k\,.
\label{symequation}
\end{equation}
%-----------------------------
Also, by \eqref{eq8.3}, 
\begin{align*}
P_k&=\nabla_{Q_k}H=\left(U_{k}P^T_{k+1}\right)^T=P_{k+1}U^T_{k}\,,
\qquad  
Q_{k+1} = \nabla_{P_{k+1}}H = Q_kU_k\,.
\end{align*}
Hence we obtain equations (\ref{discrbopteqns}).
Combining (\ref{discrbopteqns}) with (\ref{symequation}) we get
%-----------------------------
\begin{equation}
\label{equ_final_th_2.7}
U_{k}\Lambda-\Lambda U^T_{k}=P_k^TQ_k-Q_k^TP_k
\end{equation}
%-----------------------------
Now replace $P_k$ by $-P_k$ and $P_{k+1}$ by $-P_{k+1}$; thus
\eqref{discrbopteqns} remains unchanged but \eqref{equ_final_th_2.7}
is transformed to \eqref{Ukdef} which yields the stated result.
\end{proof}

We now define the symmetric representation of the
discrete rigid body equations as follows:
\begin{equation}
Q_{k+1} =Q_kU_{k}, \qquad
P_{k+1} =P_kU_{k}, \qquad   k=0, \ldots, N-1, 
\label{discrbopteqns1}
\end{equation}
where
$U_{k} \in \operatorname{SO}(n)$ is defined by
\begin{equation}
U_{k}\Lambda-\Lambda U^T_{k}=Q_k^TP_k-P_k^TQ_k.
\label{Ukdef1}
\end{equation}
We will write this as
\begin{equation}
J _D U _k = Q_k^TP_k-P_k^TQ_k
\end{equation}
where $J _D : \operatorname{SO}(n) \rightarrow \mathfrak{so}(n)$
(the discrete version of the moment of inertia operator $J$) is 
defined by $J_D U := U \Lambda - \Lambda U ^T$.
%-----------------------------
Notice that the derivative of $J_D$ at the identity in the direction
$\Omega \in \mathfrak{so}(n)$ is $J:\mathfrak{so}(n)\ni \Omega 
\mapsto \Omega \Lambda + \Lambda \Omega \in  \mathfrak{so}(n)$,
the classical moment of inertia operator on $\mathfrak{so}(n)$.
Since $J$ is invertible, $J_D$ is a diffeomorphism from a neighborhood 
of the identity in $\operatorname{SO}(n)$ to a neighborhood of $0$ 
in $\mathfrak{so}(n)$.
Using these equations, we have the algorithm
$(Q_k,P_k)\mapsto (Q_{k+1}, P_{k+1})$ defined by: compute $U_k$ from
(\ref{Ukdef1}), compute
$Q_{k+1}$ and $P_{k+1}$ using (\ref{discrbopteqns1}). Note that the
update map for
$Q$ and $P$ is done in parallel.

Equation (\ref{Ukdef1}) can be solved for $U_k$ under certain
circumstances, as discussed above, in \cite{MoVe1991}, and in 
\cite{CaLe2001}; we come back later to this issue.

As discussed in \cite{BlCrMaRa2002} these equations are equivalent on a certain 
set to the discrete Moser-Veselov equations
for the classical rigid body if we identify $U$ with the body 
angular momentum.  We shall say more about Moser-Veselov as well as equivalence in the smooth setting below.

We can now  obtain the machine learning generalization of the discrete rigid body
equations.

\begin{definition}\label{mvoptprobm1}
Let $\Lambda$ be a positive definite diagonal matrix. Let
$\overline{Q}_0, \overline{Q}_N^a\in \operatorname{SO}(n)$,
$a=1, \ldots, M$, be given and fixed.
Let
\begin{equation}
\hat{V}(U_0, \ldots, U_{N-1})
:=\sum_{k=1}^{N-1}\operatorname{trace}(\Lambda U_{k})
+\sum_{a=1}^M\phi(\overline{Q}^a_N)
\label{discrboptm}
\end{equation}
and let $\phi:\mathbb{R}^n\times\mathbb{R}^n\rightarrow\mathbb{R}^+$  be a given smooth function. 
Define the optimal control problem
\begin{equation}
\mathop{\rm min}_{U_k\in \operatorname{SO}(n)}\hat{V}(U_0, \ldots, U_{N-1})=
\mathop{\rm min}_{U_k\in \operatorname{SO}(n)}
\sum_{k=1}^{N-1}\operatorname{trace}(\Lambda U_{k})
+\sum_{a=1}^M\phi(\overline{Q}^a_N)
\label{discrbopt2m}
\end{equation}
subject to dynamics and initial data
\begin{equation}
Q^a_{r+1}=Q^a_rU_{r},
\qquad Q^a_0=\overline{Q}^a_0, \qquad r=0, \ldots, N-1,
\label{discrbeqm}
\end{equation}
for $Q_k, U_k\in \operatorname{SO}(n)$.
\end{definition}

\begin{theorem}
A solution of the optimal control problem in Defintion \ref{mvoptprobm1}
satisfies the optimal evolution  equations  for each
$a=1, \ldots, M$,
\begin{equation}
Q^a_{k+1} = Q^a_kU_{k} \qquad
P^a_{k+1} = P^a_kU_{k}\,, \qquad
k=0, \ldots, N-1,
\qquad Q^a_0 = \overline{Q}^a_0,
\label{discrbopteqnsm}
\end{equation}
where $P^a_k$ is the discrete covector in the discrete maximum principle and
$U_{k}$ is defined by
%-----------------------------
\begin{equation}
U_{k}\Lambda-\Lambda U^T_{k}=\sum_{^Ma=1}((Q^a_k)^TP^a_k-(P^a_k)^TQ^a_k)\,.
\label{Ukdef2}
\end{equation}
%-----------------------------
with 
\begin{equation}
P^a_N=\frac{\partial\phi(Q^a_N)}{\partial Q^a_N}\,.
\end{equation}

\end{theorem}

\begin{proof}
We apply Proposition \ref{gendiscpropm} with 
%-----------------------------
\begin{align}
H(P_{k+1},Q_k,U_{k})&=
\sum_{a=1}^M\operatorname{trace}((P^a_{k+1})^TQ^a_kU_{k})+\sum_{a=1}^M\phi(Q^a_N)
-\operatorname{trace}(\Lambda U_{k}) 
%&=\operatorname{trace}\left(\sum_{a=1}^M((P^a_{k+1})^TQ^a_k-\Lambda)U_{k}\right)\,.
\label{dischamm}
\end{align}

Then the computation is as above for the optimal control setting. 

\end{proof}

\subsection{Classical Moser-Veselov equations}

The dynamics above  are related to the 
 Moser-Veselov equations as discussed in \cite{BlCrMaRa2002}.
 
The Moser-Veselov equations for the discrete rigid body go back to  
\cite{Veselov1988} and \cite{MoVe1991}. Their work is closely related to the
development of variational integrators; see, e.g.,  \cite{MaPeSh1999} and
\cite{KaMaOrWe2000}. Another approach
to integrating differential equations on manifolds is discussed
in \cite{CrGr1993}. See also \cite{IsMcZa1999},
\cite{BuIs1999} and \cite{BoSu1999}. 

\paragraph{Review of the Moser-Veselov Discrete Rigid Body.} We briefly
review the \cite{MoVe1991} discrete rigid body equations. Let
$Q_k \in\operatorname{SO}(n)$ denote the rigid body configuration at 
time $k$, let $\Omega_k\in\operatorname{SO}(n)$ denote the body angular 
velocity  at time $k$, and let $M_k$ denote the body angular momentum at 
time $k$. These quantities are related by the Moser-Veselov equations
\begin{align}
\Omega_k&= Q_k^TQ_{k-1} \label{mv1}\\
M_k&= \Omega_k^T\Lambda-\Lambda\Omega_k   \label{mv2}\\
M_{k+1}&=\Omega_kM_k\Omega_k^T.\label{mv3}
% \label{drbn.eqn}
\end{align}
These equations may be viewed as
defining two different algorithms.
\paragraph{MV-Algorithm 1.} Define the step ahead map
%-----------------------------
\begin{equation}
\left(Q_k, Q_{k+1}\right)\mapsto
\left(Q_{k+1}, Q_{k+2}\right)
\end{equation}
%-----------------------------
as follows: compute $\Omega_{k+1}$ from (\ref{mv1}), compute
$M_{k+1}$ from (\ref{mv2}), compute $M_{k+2}$ from (\ref{mv3}),
compute $\Omega_{k+2}$ from (\ref{mv2}) and then compute
$Q_{k+2}$ from (\ref{mv1}).

\paragraph{Remark.} Given $M _k$, conditions under which equation
(\ref{mv2}) is solvable for $\Omega_k$ are discussed in \cite{MoVe1991}
and  \cite{CaLe2001}. 

\paragraph{MV-Algorithm 2.} Define the map:
\begin{equation}
\left(Q_k, M_{k}\right)\mapsto
\left(Q_{k+1}, M_{k+1}\right)
\end{equation}
as follows: compute $\Omega_k$ from (\ref{mv2}),  compute
$M_{k+1}$ from (\ref{mv3}), compute $\Omega_{k+1}$ from (\ref{mv2})
and compute $Q_{k+1}$ from (\ref{mv1}).

\paragraph{Discrete Variational Principle.} The Moser-Veselov
equations (\ref{mv1})-(\ref{mv3}) can be obtained by a discrete
variational principle, as was done in \cite{MoVe1991}. This
variational principle has the general form described in
discrete mechanics; see, e.g., \cite{MaWe1997}, \cite{BoSu1999}, 
and \cite{MaWe2001}.
Namely, stationary points of the functional
%-----------------------------
\begin{equation}
\hat{S}=  \sum_k \operatorname{trace}(Q_k \Lambda Q_{k+1}^T)
\label{mvl}
\end{equation}
%-----------------------------
on sequences of orthogonal $n\times n$ matrices yield the Moser-Veselov
equations. This variational approach can be justified as in 
\cite{MaPeSh1999}.

%\todo{TSR: I am here on January 6}

As mentioned above we can prove that symmetric representation of the rigid body equations and the Moser-Veselov equations 
are equivalent when restricted to a suitable set. 
It is easy to see the following: suppose that we have a solution $(Q _k, P _k) $ to the symmetric discrete rigid body equations 
We can then  produce a solution $(Q _{k + 1}, M
_{k + 1}) $ of the Moser-Veselov equations: 
\begin{equation}
M_{k + 1}=Q_k^TP_k-P_k^TQ_k
\label{Mdef1}
\end{equation}
will give us the required $M _{k + 1}$ that does the job. 
 We refer to \cite{BlCrMaRa2002} for the full proof of equivalence in the discrete setting and we shall discuss below
equivalence of  the symmetric and standard rigid body  in the smooth setting.

\section{Smooth Setting and Double Bracket Equations}  \label{optsec}

These general ideas can also be recast in the smooth setting.  We consider here the rigid body analysis followed by an 
analysis of certain flows on adjoint orbits. 

\subsection{Classical  $n$-dimensional rigid body equations}
Firstly we review the classical rigid body equations 
 in $n$ dimensions for completeness. Further details may be found in \cite{BlCrMaRa2002}.

We use the following
pairing (multiple of the Killing form) on $\mathfrak{so}(n)$, the Lie
algebra of $n \times n $ real skew matrices regarded as the Lie algebra
of the $n$-dimensional proper rotation group $\operatorname{SO}(n)$:
\begin{equation}\label{killing.eqn}
            \left\langle  \xi, \eta
            \right\rangle
=  - \frac{1}{2} \operatorname{trace} (\xi \eta).
\end{equation}
The factor of $1/2$ in (\ref{killing.eqn}) is to make this inner product
agree with the usual inner product on $\mathbb{R}^3$ when it is
identified  with $ \mathfrak{so}(3)$ in the following standard way: 
associate the $3 \times 3 $ skew matrix $\hat{u }$ to the vector $u$ by
$\hat{u } \cdot v = u \times v $, where $u \times v $ is the usual
cross product in ${\mathbb R}^3$.

We use this inner product to identify the dual of the Lie algebra,
namely
$\mathfrak{so}(n)^\ast$, with the Lie algebra $\mathfrak{so}(n)$.

We recall from \cite{Manakov1976} and \cite{Ratiu1980} that the left
invariant generalized rigid body equations on
$\operatorname{SO}(n)$ may be written as
\begin{equation}
\dot Q = Q\Omega ; \qquad
\dot M = [M,\Omega]\,, %\tag{RBn}
\label{rbl}
\end{equation}
where $Q\in \operatorname{SO}(n)$ denotes the configuration space
variable (the attitude of the body), $\Omega=Q^{-1}\dot{Q} \in
\mathfrak{so}(n)$ is the body angular velocity, and
\[
M:=J(\Omega)=\Lambda\Omega +\Omega\Lambda \in
\mathfrak{so}(n)
\]
            is the body angular momentum. Here
$J: \mathfrak{so}(n) \rightarrow  \mathfrak{so}(n) $ is the symmetric
(with respect to the inner product (\ref{killing.eqn})), positive definite,
and hence invertible, operator defined by
\[
J(\Omega)=\Lambda\Omega +\Omega\Lambda ,
\]
            where $\Lambda$ is
a diagonal matrix satisfying $\Lambda_i + \Lambda_j >0$ for
all $i \neq j$. For $n=3$ the elements of $\Lambda_i$
are related to the standard diagonal moment of inertia tensor $I$ by
$I_1 = \Lambda_2 + \Lambda_3$,  $I_2 = \Lambda_3 + \Lambda_1$,
            $I_3 = \Lambda_1 + \Lambda_2$.

The equations $ \dot{ M } =  [ M, \Omega
] $ are readily checked to be the Euler-Poincar\'e equations on
$\mathfrak{so}(n)$ for the Lagrangian
$
l ( \Omega ) = \frac{1}{2}  \left\langle  \Omega , J
( \Omega )
\right\rangle .
$
This corresponds to the Lagrangian on $T \operatorname{SO}(n) $ given by
\begin{equation} \label{RBLag_group.eqn}
L ( g , \dot{g}) = \frac{1}{2} \left\langle g ^{-1} \dot{g}, J ( g ^{-1}
\dot{g} ) \right\rangle\,.
\end{equation}

We note  that the dynamic rigid body 
equations on $\operatorname{SO}(n)$  and indeed on any semisimple Lie 
group are integrable (\cite{MiFo1976}). A key observation in this 
regard, due to Manakov, was that one could write the  generalized 
rigid body equations as Lax equations with parameter: 
\begin{equation} \frac{d}{dt}(M+\lambda \Lambda^2)= [M+\lambda 
\Lambda^2,\Omega+\lambda \Lambda], \label{lambda_eqn} \end{equation} 
where $
M=J(\Omega)=\Lambda\Omega +\Omega \Lambda
$, as in \S2.
The nontrivial coefficients of $\lambda$ in the
traces of the powers of $M+\lambda \Lambda^2$ then yield
the right number of independent integrals in
involution to prove integrability of the flow on a
generic adjoint orbit of $\operatorname{SO}(n)$ (identified with the
corresponding coadjoint orbit). Useful references are
\cite{Bogayavlenski1994} and \cite{FeKo1995}.)
\cite{MoVe1991} show that there is a
corresponding formulation of the discrete rigid body equations
with parameter.

\subsection{Smooth optimal control and the symmetric rigid body equations}
Now we briefly review, see  \cite{BlCr1996} and \cite{BlBrCr1997}, two
results which link  the theory of optimal control with the rigid body equations. 
\begin{definition}\label{rboptcontprob}

            Let $T >0 $, $Q _0, Q _T \in \operatorname{SO}(n)$
be given and fixed. 
Let the rigid body optimal control problem be given by
\begin{equation}
\mathop{\rm min}_{U\in
\mathfrak{so}(n)} \frac{1}{4}\int_0^T
\langle U,J(U)\rangle dt
\label{optr}
\end{equation}
subject to the constraint on $U$ that there be a curve
$Q (t) \in \operatorname{SO}(n)$ such that
\begin{equation}
\dot Q=QU\qquad Q(0)=Q_0,\qquad Q(T)=Q_T.
\label{eqnr}
\end{equation}

\end{definition}

\begin{proposition} The rigid body optimal control problem
has optimal evolution equations

\begin{equation}\label{srb1}
\dot{Q}=QU\qquad \dot{P}=PU
\end{equation}

 where $P$ is the costate vector given by the maximum
principle.

The optimal controls in this case are given by
\begin{equation}
U=J^{-1}(Q^TP-P^TQ).
\end{equation}
\end{proposition}

\paragraph{Remark.}
The proof (see \cite{BlCr1996}) simply involves, as in the discrete analysis above,
writing the Hamiltonian of the maximum principle as
\begin{equation}
H= \left\langle P,QU \right\rangle +\frac{1}{4} \left\langle
U,J(U)
\right\rangle,
\end{equation}
where the costate vector $P$ is a multiplier enforcing the
dynamics, and then maximizing with respect to $U$ in the standard
fashion (see, for example, Brockett [1973]). 

 We refer to the equations (\ref{srb1}) as the {\it symmetric representation of the rigid body 
 equations}. We can now recover the classical rigid body equations:

\begin{proposition}\label{SRBtoRB.prop}
If $(Q, P)$ is a solution of (\ref{srb1}), then $(Q, M) $ where
$M = J (\Omega)$, $\Omega=U$, and $\Omega = Q ^{-1}  \dot{Q}$ satisfies
             the rigid body equations (\ref{rbl}).
\end{proposition}

\begin{proof} Differentiating  $M=Q^TP-P^TQ$ and using the
equations (\ref{srb1}) gives the second of the equations
(\ref{rbl}).
\end{proof}

While in general there are no
constraints on the costate vector $P\in\mathfrak{gl}(n)$
one can consider the restriction of the extremal flows to 
invariant submanifolds. This limits possible extremal 
trajectories that can be recovered. For example
this system restricts to a system on $\operatorname{SO}(n)\times
\operatorname{SO}(n)$. One can make other assumptions on
the costate vector. For example, suppose we assume a costate
vector $B$ such that $Q^TB$ is skew. Then it is easy to check
that that the extremal evolution equations become
%-----------------------------
\begin{equation}
\dot Q = QJ^{-1}(Q^TB); \qquad
\dot B = BJ^{-1}(Q^TB)\,,
\label{rbnlms}
\end{equation}
%-----------------------------
and that these equations restrict to an invariant submanifold defined
by the condition that $Q^TB$ is skew symmetric.
These are the McLachlan-Scovel equations (\cite{McSc1995}).
%Comparing these equations with (\ref{rbnl}) we see that
%$B=P-QP^TQ.$ 
%    There is a similar esult for the right invariant case.
\medskip

We can now generalize to the machine learning setting: 

\begin{definition}\label{rboptcontprobm}
            Let $T >0 $, $Q _0, Q _T \in \operatorname{SO}(n)$
be given and fixed. 
Let the rigid body optimal control problem be given by
\begin{equation}
\mathop{\rm min}_{U\in
\mathfrak{so}(n)} \frac{1}{4}\int_0^T
\langle U,J(U)\rangle dt +\sum_{a=1}^M\phi(Q^a_T)
\label{optrm}
\end{equation}
subject to the constraint on $U$ that there be a curve
$Q (t) \in \operatorname{SO}(n)$ such that
\begin{equation}
\dot Q^a=Q^aU\qquad Q^a(0)=Q^a_0,\, a=1\dots M.
\label{eqnrm}
\end{equation}

\end{definition}

\begin{proposition}  The smooth machine learning symmetric rigid body flow
has optimal evolution equations

\begin{equation}
\dot Q^a=Q^aU,\, \dot P^a=P^aU
\end{equation}

where $P$ is the costate vector given by the maximum
principle

The optimal controls in this case are given by
\begin{equation}
U=\sum_aJ^{-1}((Q^a)^TP^a-(P^a)^TQ^a).
\end{equation}
\end{proposition}

and we have the endpoint conditions

\begin{equation}
P^a_T=\frac{\partial\phi (Q^a_T)}{\partial Q^a_T}
\end{equation}

\subsection{Local equivalence of classical rigid body and
symmetric rigid body equations.}

Above we saw that solutions of the symmetric rigid body equations
can be mapped to solutions of the rigid body system. As in \cite{BlCrMaRa2002} we can
consider the converse question. Thus, suppose we have a solution
$(Q, M) $ of the standard left invariant rigid body equations.
We seek to solve for $P$ in the
expression
\begin{equation} \label{M_Q_P.eqn}
M=Q^TP-P^TQ.
\end{equation}

For the following discussion, it will be convenient to make use of the
operator norm on matrices. Recall  that
this norm is given by
$
\| A \|_{\rm op} = \sup \left\{ \| A x \| \mid \| x \| = 1 \right\}, 
$
where the norms on the right hand side are the usual Euclidean
space norms.

Since elements of $\operatorname{SO}(n) $ have
operator norms bounded by $1$ and since the operator norm
satisfies
$\| A B \| _{\rm op} \leq \| A \| _{\rm op} \| B \| _{\rm op} $,
we see that {\it if $M$ satisfies $M=Q^TP-P^TQ$, then
$\| M \| _{\rm op} \leq 2$.} Therefore,  $\| M \| _{\rm op} \leq 2$
{\it is a necessary condition for solvability of (\ref{M_Q_P.eqn}) for
$P$.}

\begin{definition}\label{CandS.definition}
Let $C$ denote the set of $(Q,P)$ that map to
$M$'s with operator norm equal to 2 and let $S$ denote the set of
$(Q,P)$ that map to $M$'s with operator norm strictly less than 2.
Also denote by $S_M$ the set of points $(Q,M)
\in T^*\operatorname{SO}(n)$ with
$\| M \| _{\rm op} < 2$. For the left invariant system we trivialize
$T^*\operatorname{SO}(n) \cong \operatorname{SO}(n) \times
\mathfrak{so}(n)^\ast$ by means of left translation to the identity and
we identify $\mathfrak{so}(n)^\ast $ with $\mathfrak{so}(n)$ using the
Killing metric (\ref{killing.eqn}).
\end{definition}

Note that $C$ contains pairs $(Q,P)$ with the property that $Q^TP$ is
both skew and orthogonal.

Recall that $\sinh : \mathfrak{so}(n) \rightarrow
\mathfrak{so}(n)$ is defined by
$
\sinh \xi = \left( e ^\xi - e ^{- \xi }  \right) /2 $.
One sees that indeed $\sinh $ takes values in $\mathfrak{so}(n)$
by using, for example, its series expansion:
\[
\sinh \xi = \xi + \frac{1}{3!}\xi ^3 + \frac{1}{5! } \xi ^5 + \ldots.
\]
Recall from calculus that the inverse function $\sinh ^{-1} (u)$ has a
convergent power series expansion for $| u |  < 1 $ that is given by
integrating the power series expansion of the function
$1/ \sqrt{1 + u ^2 }$ term by term. This power series expansion
shows that the map $\sinh : \mathfrak{so}(n) \rightarrow
\mathfrak{so}(n)$ has an inverse on the set $U = \left\{ u \in
\mathfrak{so}(n) \mid \| u \| _{\rm op} < 1 \right\}$. We shall denote
this inverse by $\sinh ^{-1}$, so
$
\sinh ^{-1}: U \rightarrow \mathfrak{so}(n).
$

\begin{proposition} For $\| M \| _{\rm op} < 2 $, the equation(\ref{M_Q_P.eqn})
has the  solution
\begin{equation}\label{Pequ}
P=Q\left( e^{\sinh^{-1}M/2}\right)
\end{equation}
\end{proposition}
\begin{proof} Notice that
$
M=e^{\sinh^{-1}M/2}-e^{-\sinh^{-1}M/2}\,.
$
\end{proof}

\begin{proposition} The sets $C$ and $S$ are invariant under the double
rigid body equations.
\end{proposition}

\begin{proof}Notice that the operator norm is invariant under
conjugation; that is, for $Q \in \operatorname{SO}(n)$ and
$M \in \mathfrak{so}(n)$, we have
$
\| Q M Q ^{-1} \| _{\rm op} = \| M \| _{\rm op}.
$
This is readily checked from the definition of the operator norm.
Recall that under the identification of the dual
$\mathfrak{so}(n)^\ast$ with the space $\mathfrak{so}(n)$, the
coadjoint action agrees with conjugation. Thus, the map
$f: \mathfrak{so}(n) \rightarrow \mathbb{R}$; $M
\mapsto \| M \|_{\rm op}$ is a Casimir function and so is invariant
under the dynamics. In particular, its level sets are invariant and
so the sets $S$ and $C$ are invariant. \end{proof}
\medskip

\paragraph{The Hamiltonian form of the symmetric rigid body equations.}
Recall that the classical rigid body equations are Hamiltonian
on $T^*\operatorname{SO}(n)$ with respect to the canonical symplectic
structure on the cotangent bundle of $\operatorname{SO}(n)$. The
following result gives the corresponding theorem for the symmetric case. The proof 
is given in \cite{BlCrMaRa2002}

\begin{proposition}
Consider the
Hamiltonian system on the symplectic vector space $ \mathfrak{gl}(n)
\times
\mathfrak{gl}(n)$ with the symplectic structure
\begin{equation}
\Omega _{\mathfrak{gl}(n)} (\xi_1, \eta _1, \xi_2, \eta _2 )
= \frac{1}{2} \operatorname{trace} ( \eta _2 ^T \xi _1 -
\eta _1 ^T \xi _2 )
\label{gln_symp}
\end{equation}
where $(\xi_i,\eta_i)\,,i=1,2$ are elements of  $ \mathfrak{gl}(n)
\times \mathfrak{gl}(n)$
and Hamiltonian
\begin{equation}
H ( \xi, \eta ) = - \frac{1}{8} \operatorname{trace}
\left[ \left( J^{-1}(\xi^T \eta -\eta^T \xi ) \right) \left( \xi^T\eta -
\eta ^T \xi \right) \right] .
\label{ourHam}
\end{equation}
The corresponding Hamiltonian system leaves $\operatorname{SO}(n)
\times \operatorname{SO}(n) $ invariant and induces on it, the
flow of the symmetric representation of the rigid body system.
\end{proposition}
Note that the above Hamiltonian is equivalent to the standard rigid body Hamiltonian 
$
H=\frac{1}{4} \left\langle J^{-1}M,M\right\rangle,
$
as in \cite{Ratiu1980}.

\subsection{Optimality on adjoint orbits and  learning}

These general ideas can also be extended to a 
 variational problem on the adjoint 
orbits of compact Lie groups as in \cite{BlCr1996}.

Let $\frak g$ be a complex semisimple Lie algebra, $\frak g_u$ its compact
real form, and $G_u$ the corresponding compact group.
In this case a natural drift free control system on an orbit
of $G_u$ takes the form
\begin{equation}
\dot x=[x,u]
\label{orb}
\end{equation}

We remark that we formulate the problem in this generality for convenience,
but the most useful case to bear in mind is the algebra $\mathfrak{su}(n)$
of skew-Hermitian matrices or the algebra $\mathfrak{so(}n)$ of skew symmetric
matrices (the intersection of the compact and normal real forms of
the the algebra $\mathfrak{sl}(n, \Bbb C)$). Orbits in this case are similarity
orbits under the group action.

We then consider the following generalization
of the functional suggested by Brockett [1994] (we shall return to
Brockett's precise problem shortly):
\begin{equation}
\eta(x,u)=\int_0^{t_f}1/2||u||^2-V(x)dt
\label{var}
\end{equation}

where $||\cdot ||=<\cdot ,\cdot >^{1/2}$ is the
norm induced on $\frak g_u$ by the negative of the
Killing form $\kappa (\cdot, \cdot)$ on $\frak g$ and V is a smooth
function on $\frak g_u$. 
The pairing between
vectors $x$ in $\frak g$ and dual vectors $p$ in $\frak g^*$ may be 
written $<p,x>=-\kappa(x,p)$.

We have
\begin{theorem}
The equations of the maximum principle for the variational problem with 
functional \ref{var} subject to the dynamics \ref{orb} are
\begin{eqnarray}
\dot x&=&[x,[p,x]] \nonumber \\
\dot p&=&[p,[p,x]]-V_x\,.
\label{op}
\end{eqnarray}
\end{theorem}

\noindent {\bf Proof.} The Hamiltonian is given by
\begin{equation}
H(x,p,u)=<p,[x,u]>-1/2||u||^2+V(x)\,.
\end{equation}
Hence 
\[{\partial H\over\partial u}=-<[x,p],\cdot >-<u,\cdot >\]
and thus the optimal control is given by
\begin{equation}
u^*=[p,x]
\end{equation} 

Substituting this into $H$ we find the Hamiltonian evaluated
along the optimal trajectory is given by
\begin{equation}
H^*(p,x)=-1/2<x,[p,[p,x]]>+V(x)
\end{equation}
Computing 
\[\dot x=\left({\partial H^*\over \partial p}\right)^T\]
and
\[\dot p=-\left({\partial H^*\over \partial x}\right)^T\]
gives the result.\quad $\blacksquare$

A particularly interesting special case of this problem
is that of Brockett [1994] where we have
\begin{corollary}
The equations of the 
maximum principle for the variational problem \ref{var} 
subject to equations \ref{orb} with $V(x)=-\tfrac{1}{2} \|[x,n]\|^2$ are
\begin{eqnarray}
\dot x&=&[x,[p,x] \nonumber \\
\dot p&=&[p,[p,x]]-[n,[n,x]]\,.
\label{opb}
\end{eqnarray}
\end{corollary}

The proof of the corollary follows immediately, setting
$V(x)=\tfrac{1}{2} \left\langle x,[n,[n,x]]\right\rangle$. Note that with this functional the
equations lie naturally on an adjoint orbit. In addition, these
equations are interesting in that the optimal flow may be
related to the integrable Toda lattice equations (see
\cite{Bloch1990}, \cite{BlBrRa1992} and Brockett [1994].)

The smooth machine learning version of this problem considers for smooth functions $\phi^a:\mathbb{R}\times\mathbb{R}\rightarrow\mathbb{R}^+$, $a=1,\dots, M$
\begin{equation}
\eta(x,u)=\tfrac{1}{2} \int_0^{T}\|u\|^2-V(x)dt +\sum_{a=1}^M\phi(x^a_T)
\label{varm}
\end{equation}
with $\dot{x^a}=[x^a,u],\,a=1\dots M$

As before we now have 
\begin{corollary}
The equations of the 
maximum principle for the variational problem 
subject to equations \ref{orb} with $V(x)=-1/2||[x,n]||^2$ are
\begin{eqnarray}
\dot x^a&=&[x^a,[p^a,x^a]] \nonumber \\
\dot p^a&=&[p^a,[p^a,x]]-[n,[n,x^a]]\,.
\label{opbm}
\end{eqnarray}

with 
\begin{equation}
p^a_T=\frac{\partial\phi(x^a_T)}{\partial x^a_T}
\end{equation}

\end{corollary}

In particular  we would like to consider $\phi(x^a_T)=<x^a_T,n>$.

Then  we can see the final solution tends to  the double bracket equation $\dot{x}=[x,[x,n]]$ and the second term in the costate equation 
goes to zero. 

One can then write a discrete version of these equations using an appropriate discrete algorithm and following our formalism above. 
This will be considered in future work. 

\section{Conclusions} \label{conclusions.section}

In this paper we have discussed discrete optimal control systems 
and related them to equations for machine learning. In particular, 
we have considered the symmetric 
formulation of the rigid body equations, both discrete and smooth, and 
discussed double bracket equations. 

We note also that the  analysis here can be extended to other 
systems, such as the  full Toda dynamics.  We intend to discuss such 
extensions and numerical aspects, as well as the subRiemannian and quantum settings,  in a future publication. 

\paragraph{Acknowledgement.} We acknowledge the inspiration and guidance 
of Roger Brockett over many years and the wonderful 
collaborations we had together.  We would also like to thank Fred Leve for his support for the meeting in which Roger was honored
and for his support for research in nonlinear control in general, and to thank Maani Ghaffari for his valuable comments on the manuscript.

\end{document}